\documentclass[10pt,a4paper]{article}
\usepackage{geometry}
\usepackage[english]{babel}
\usepackage{amsmath}
\usepackage{graphicx}
\usepackage{ marvosym }
\usepackage{tikz}
\usetikzlibrary{arrows}
\usepackage[utf8]{inputenc}
\usepackage{mathtools}
\usepackage{geometry}
\usepackage{amsfonts}
\usepackage{amssymb}
\usepackage{amsthm}
\usepackage{thmtools}
\usepackage{t1enc}
\usepackage[titles]{tocloft}
\usepackage{makeidx}
\usepackage{wasysym}
\usepackage{stmaryrd}
\usepackage{calc}  
\usepackage{enumitem} 
\usepackage[refpage]{nomencl}

\usepackage[colorlinks=true,urlcolor=blue, linkcolor=blue,pageanchor=false ]{hyperref}
\usepackage{algpseudocode}
\usepackage{tikz}
\usetikzlibrary{arrows}
\usepackage{float}

\DeclarePairedDelimiterX\Set[2]{\lbrace}{\rbrace}%
 { #1 \,\delimsize:\, #2 }

\theoremstyle{definition}

\theoremstyle{plain}
\newtheorem{tetel}{Theorem}
\newtheorem{all}[tetel]{Proposition}

\newtheorem*{all*}{Proposition}
\newtheorem*{seged*}{Segédállítás}

\newtheorem{lem}[tetel]{Lemma}

\newtheorem*{lem*}{Lemma}
\theoremstyle{definition}

\newtheorem*{defi*}{Definition}

\newtheorem{fel*}[tetel]{Feladat*}

\newtheorem*{megf*}{megfigyelés}
\theoremstyle{remark}
\newtheorem{megj}[tetel]{Remark}
\newtheorem*{megj*}{Remark}

\newcommand{\defeq}
{\stackrel{\text{def}}{=}}

\newenvironment{biz}{\par\noindent{\itshape Proof:}\ }{\rule{1.5ex}{1.5ex}}
\newenvironment{sbiz}{\par\noindent{\itshape Proof:}\ }{\newmoon}

\title{Highly connected infinite digraphs without edge-disjoint back and forth paths between a certain vertex pair}
\author{Attila Joó \thanks{MTA-ELTE Egerváry Research Group, Department of Operations Research, Eötvös Loránd University,
Budapest, Hungary.
 Email: {\tt joapaat@cs.elte.hu} }}

\date{2015}

\begin{document}
\maketitle
This is the peer reviewed version of the following article: \cite{joo2016highly}, which has been published in final form at 
\url{http://dx.doi.org/10.1002/jgt.22046}. This article may be used for non-commercial purposes in accordance with Wiley Terms and 
Conditions for Self-Archiving.

\begin{abstract}
\noindent We construct for all $ k\in \mathbb{N} $  a $ k $-edge-connected digraph $ D $ with $ s,t\in V(D) $ such 
that there are no edge-disjoint $ s \rightarrow t $ and $ t\rightarrow s $  paths. We use in our construction ``self-similar'' graphs 
which technique could be useful in other problems as well.
\end{abstract}

\section{Introduction}
\subsection{Basic notions}
In this paper by ``path'' we  mean a finite, simple, directed path. Sometimes we define a 
path 
of 
a digraph $ D=(V,A) $ by a 
finite sequence $ v_0,\dots,v_n $ of vertices of $ D $. If there are more than one edges from $ v_i $ to $ v_{i+1} $ for some $ 
i<n $, then it is not specified which edge is used by the path, so we use this kind of definition only if it does not 
matter. An $ u\rightarrow v $ path is a path  with initial vertex $ u $ and terminal vertex $ v $. Its length is the number of its 
edges. We call a digraph $ D $ connected if for all $ u,v\in V(D) $ there is a $ u\rightarrow v $ path in $ D $. For $ U\subseteq V $ 
let $ \mathsf{span}_D(U) $ be the set of those edges of $ D $ whose heads and tails are contained in  $U $ and let $ 
D[U]= (U,\mathsf{span}_D(U)) $. If it is clear what digraph we talk about, then we omit the subscripts.\\

\subsection{Background and Motivation}

R. Aharoni and C. Thomassen proved by a construction the following theorem that shows that several theorems about 
edge-connectivity properties of 
finite graphs and digraphs become ``very'' false in the infinite case.

\begin{tetel}[R. Aharoni, C. Thomassen \cite{aharoni1989infinite}]
For all $ k\in \mathbb{N}$ there is an infinite graph $ G=(V,E) $ and $ s,t\in V $ such that $ E $ has a $ k $-edge-connected 
orientation but 
for each path $ P $ between $ s $ and $ t $ the graph $ G=(V,E\setminus E(P)) $ is not connected.
\end{tetel}

In this article we would like to introduce a similar result.
If $ D $ is a $ k $-edge-connected finite digraph, then for all $ s_1,t_1,\dots ,s_k,t_k\in V(D) $ there are pairwise edge-disjoint paths 
$P_1,\dots ,P_k $ such that $ P_i $ is an $ s_i\rightarrow t_i $ path. This fact is implied by the following Theorem of W. Mader as well 
as the 
(strong form of) Edmonds' Branching theorem (see \cite{frank2011connections} p. $ 349 $ Theorem 
$ 10.2.1 $).

\begin{tetel}[W. Mader \cite{mader1981property}]
Let $ D=(V,A) $ be a $ k+1 $-edge-connected, finite digraph and $ s,t\in V $. Then there is an $ s\rightarrow t $ path $ P $ 
such that $ 
(V,A\setminus A(P)) $ is $ k $-edge-connected.
\end{tetel}

 We will show that in the infinite case there is no $ k\in \mathbb{N} $ such that $ k $-edge-connectivity guarantees even 
the existence of edge-disjoint $ s_1 \rightarrow t_1 $ and $ s_2 \rightarrow t_2 $ paths for all $ s_1, t_1, s_2, t_2 $ vertices. Not 
even in the special case where the two ordered vertex pair is the reverse of each other.

\section{Main result}

\begin{tetel}\label{result}
For all $ k\in \mathbb{N} $  there exists a $ k $-edge-connected digraph without back and forth edge-disjoint 
paths between a certain vertex pair. 
\end{tetel}
\begin{proof}

\noindent Let  $ k\geq2 $ be fixed, $ I=\{ 0,\dots,2k-1 \}$, $ I_e=\{ i\in I: i\text{ is even } \} $, $ I_o=I\setminus I_e $. 
Denote by $ I^{*} $ the set of finite sequences from $ I $. Let the vertex set $ V $ of the digraph  is the union of the 
disjoint sets 
 $  \{  s_\mu :  \mu\in I^{*} \} $  ( we mean $ s_\mu=s_\nu $ iff $ \mu=\nu $) and $ \{  t_\mu :  \mu\in I^{*} \} $ ( $ 
t_\mu=t_\nu $ iff $ \mu=\nu $).  If $ \mu $ is the empty sequence we write simply $ s,t $ and we denote the 
 concatenation of sequences by writing them successively. For $ \nu\in I^{*} $ let denote the set $ \{ r_{\nu\mu}: 
 r\in \{ s,t \},\ \mu\in I^{*} \}\subseteq V $ by $  V_\nu $. The edge-set $ A $  of the digraph consists of the following edges.
For all $ \mu \in I^{*} $ there are $ k $ edges in both directions  between the two elements of the following pairs: $\{ s_\mu, t_{\mu1} 
\},\   
\{ s_{\mu  i}, t_{\mu  (i+2)}\}\ (i=0,\dots,2k-3),\   \{ s_{\mu(2k-2)}, t_\mu  \}   $. Simple directed edges are $ (s_\mu,  
t_{\mu0} ), (t_{\mu i}, s_{\mu (i+1)} )_{i\in I_e}, (s_{\mu i},t_{\mu (i+1)})_{i\in I_o\setminus \{ 2k-1 \}},$  
$(s_{\mu(2k-1)},t_\mu)\ $ for all $ \mu\in I^{*} $. Finally $ D\defeq (V,A) $ (see figure \ref{önhas k=3}). 

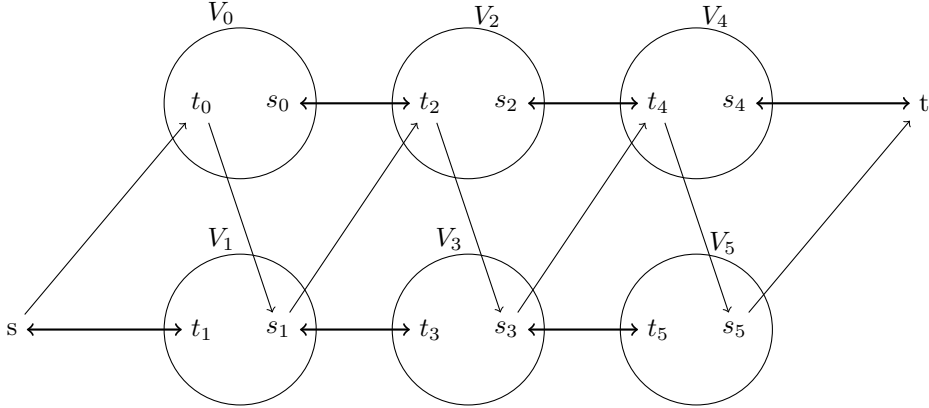
\begin{figure}[h]
\centering

\begin{tikzpicture}
\node (s)  at (-6,0) {s};
\node (t_1) at (-3.5,0) {$t_1$};

\node (s_1) at (-2.5,0) {$s_1$};
\node (t_3) at (-0.5,0) {$t_3$};

\node (s_3) at (0.5,0) {$s_3$};
\node (t_5) at (2.5,0) {$t_5$};

\node (s_5) at (3.5,0) {$s_5$};
\node (t_0) at (-3.5,3) {$t_0$};

\node (s_0) at (-2.5,3) {$s_0$};
\node (t_2) at (-0.5,3) {$t_2$};

\node (s_2) at (0.5,3) {$s_2$};
\node (t_4) at (2.5,3) {$t_4$};

\node (s_4) at (3.5,3) {$s_4$};
\node (t) at (6,3) {t};
\draw  (s) edge[<->,thick] (t_1);
\draw  (s_1) edge[<->,thick] (t_3);
\draw  (s_3) edge[<->,thick] (t_5);
\draw  (s_0) edge[<->,thick] (t_2);
\draw  (s_2) edge[<->,thick] (t_4);
\draw  (s_4) edge[<->,thick] (t);
\draw  (s) edge[->] (t_0);
\draw  (t_0) edge[->] (s_1);
\draw  (s_1) edge[->] (t_2);
\draw  (t_2) edge[->] (s_3);
\draw  (s_3) edge[->] (t_4);
\draw  (t_4) edge[->] (s_5);
\draw  (s_5) edge[->] (t);
\draw  (-3,0) node {} circle (1);
\draw  (0,0) node {} circle (1);
\draw  (3,0) node {} circle (1);
\draw  (-3,3) node {} circle (1);
\draw  (0,3) node {} circle (1);
\draw  (3,3) node {} circle (1);
\node [outer sep=0,inner sep=0,minimum size=0]  at (-0.25,1.20) {$\Huge{V_3}$};
\node [outer sep=0,inner sep=0,minimum size=0]  at (3.35,1.15) {$\Huge{V_5}$};
\node [outer sep=0,inner sep=0,minimum size=0]  at (-3.25,1.20) {$\Huge{V_1}$};
\node [outer sep=0,inner sep=0,minimum size=0]  at (-3.25,4.20) {$\Huge{V_0}$};
\node [outer sep=0,inner sep=0,minimum size=0]  at (0.25,4.15) {$\Huge{V_2}$};
\node [outer sep=0,inner sep=0,minimum size=0]  at (3.25,4.15) {$\Huge{V_4}$};
\end{tikzpicture}
\caption{The digraph $ D $  in the case $ k=3 $. Thick, 
two-headed arrows stand for $ k $ parallel edges in both directions. The (just partially drawn) $ D[V_i] $'s are isomorphic 
to the whole $ D $ by 
Proposition \ref{isomorph}.} \label{önhas k=3}
\end{figure}
\begin{megj}
One can avoid using parallel edges (without losing the desired properties of the digraph) by dividing each of these edges 
with one-one new vertex and drawing between them $ k (k-1) $-many new directed edges, one-one for each ordered pair. One can also 
achieve $ k $-connectivity instead of $ k $-edge-connectivity by using some similarly easy modification. 
\end{megj}

\begin{all}\label{isomorph}
  For $ \nu\in I^{*} $ the function $ 
 f_\nu:V\rightarrow V_\nu,\  f_\nu(r_\mu)\defeq r_{\nu\mu}\ (r\in \{ s,t \}) $ is an isomorphism between $ D $ 
and $ D[V_\nu] $. 
\end{all}
\begin{sbiz}
It is a direct consequence of the definition of the edges since the number of edges from $ r_\mu $ to $ r'_{\mu'} $ are the 
same as from $ r_{\nu \mu} $ to $ r'_{\nu 
\mu'}  $ for all $ r,r'\in \{ s,t \},\ \nu,\mu,\mu'\in I^{*} $.
\end{sbiz}

\begin{all}\label{Dv}
Denote by $ D_v $ the digraph that we obtain from $ D $ by contracting for all $ i\in I $ the set $ V_i $  to a vertex $ v_i $. Then $ 
D_v $ 
is $ k $-edge-connected.
\end{all}
\begin{sbiz}
\noindent In the vertex-sequence $ s,v_1, v_3,\dots,v_{2k-1} $ there are $ k $ edges in both directions between 
the 
neighboring vertices such as in the 
sequence $ v_0,v_2,\dots,v_{2k-2},t $. Finally there are in both directions at least $ k $ edges between the vertex sets of the 
sequences above.
\end{sbiz}\\

For $ u\neq v $ we denote by  $ \lambda(u,v) $  the local edge-connectivity from $ u $ to $ v $ in $ D $ (i.e. 
$\lambda(u,v)= 
\min \{ 
\left|A'\right|: 
A'\subseteq A,\text{ there is no path from }
u\text{ to }v\text{ in }(V,A\setminus A') \} $) and let $ \lambda \{ 
u,v \}\defeq \min \{ 
\lambda(u,v), \lambda(v,u) \} $.
\begin{all}\label{D connected}
$ D $ is  connected.
\end{all}
\begin{sbiz}
We will show that $ \lambda\{ s,r_{\mu} \}\geq 1 $ for all $ r\in \{ s,t \},\ \mu\in I^{*} $. We will use induction on length of 
$ \mu $ (which is denoted by $ \left|\mu\right| $). Consider first the $ \left|\mu\right|=0,1 $ cases directly.

The path $ 
s,t_0,s_1,t_2,s_3,\dots, t_{2k-2},s_{2k-1},t $ shows that $ \lambda(s,t)\geq 1 $. Using the isomorphism $ f_i $  (see 
Proposition \ref{isomorph}) we may fix 
an $ s_i \rightarrow t_i $ path $ P_{s_i,t_i} $ in $ D[V_i] $ for  all $ i\in I $. The path 
 \begin{align*}
 &  t,P_{s_{2k-2},t_{2k-2}},\dots, P_{s_{2k-2j},t_{2k-2j}},\dots, P_{s_{0},t_{0}},P_{s_{1},t_{1}},s\\
 \end{align*} justifies that $ \lambda(t,s)\geq 1 $ (thus $ \lambda\{ s, t \}\geq 1 $). Then we may fix 
 a $ t_i \rightarrow s_i $ path $ P_{t_i,s_i} $ in $ D[V_i] \ (i\in I) $. The paths  \begin{align*}
& s,P_{t_1,s_1},P_{t_3,s_3},\dots,P_{t_{2j+1},s_{2j+1}},\dots,  P_{t_{2k-1},s_{2k-1}}\\   
& P_{s_{2k-1},t_{2k-1}}, P_{s_{2k-3},t_{2k-3}},\dots,P_{s_{2k-1-2j},t_{2k-1-2j}},\dots, P_{s_1,t_1},s
 \end{align*}   certify that $ \lambda \{ 
 s,r_i \}\geq 1 $ if $ r\in \{ s,t \},\ i\in I_o $. The paths  
 \begin{align*}
& t,P_{s_{2k-2},t_{2k-2}},P_{s_{2k-4},t_{2k-4}},\dots,P_{s_{2k-2-2j},t_{2k-2-2j}}\dots, P_{s_{0},t_{0}}\\  & P_{t_{0},s_{0}}, 
 P_{t_{2},s_{2}},\dots,P_{t_{2j},s_{2j}},\dots,P_{t_{2k-2},s_{2k-2}},t
 \end{align*}   certify that $ \lambda \{ 
 t,r_i \}\geq 1 $ if $ r\in 
 \{ 
  s,t \}\geq 1,\ i\in I_e $ and thus (by $ \lambda\{ s, t \}\geq 1 $ and  by  transitivity) $ \lambda \{ s,r_i \}\geq 1 $ if $ 
  r\in \{ s,t \},\ i\in I_e 
  $. Hence the cases $\mu \in I*$ with $|\mu | \leq 1$ are settled. 
  
Let be $ l\geq 1 $ and suppose  $ \lambda\{ s,r_\mu \}\geq 1 $ if $r\in \{ s,t \},\ \mu\in 
I^{*},\ 
\left|\mu\right|\leq l   $. Let  $ \nu=\mu i $, where $ i\in I $ and $ \left|\mu\right| =l$. By the induction hypothesis we 
have $ \lambda\{ s,s_\mu \}\geq 1 $.  By the 
induction hypothesis for $ l=1 $ we have $ \lambda\{ s, r_i \}\geq 1 $ and so $ \lambda\{ s_\mu, r_{\mu i } \}\geq 1 $ by the 
isomorphism $ f_\mu $. Combining these, we get  $ \lambda \{ 
s,r_{\mu i} \}\geq 1 $.
\end{sbiz}
 
\begin{lem}\label{k-edge-connected}
$ D $ is $ k $-edge-connected.
\end{lem}
\begin{biz}
Let  $ k>l\geq 1 $.
\begin{all}\label{oroklod}
Let $ \mu \in I^{*} $ arbitrary. If we delete at most $ l $ edges of the digraph $ D[V_\mu] $ in such a way that its subgraphs $ 
D[V_{\mu 
i}]\ (i\in I) $ remain connected after the deletion, then $ D[V_\mu] $ also remains connected after the deletion.
\end{all}

\begin{sbiz}
Because the isomorphism $ f_{\mu} $ it is enough to deal with the case where $\mu  $ is the empty sequence.
Denote by $ D' $ the digraph that we have after the deletion. Let $ D'_v $ be the digraph that we get from $ D' $ by contracting 
the sets $ V_i 
$  to a vertex $ v_i $ for all $ i\in I $.  The digraphs $ D'[V_i]\ (i\in I) $ are connected by assumption, thus $ D' $ is connected iff 
$ 
D'_v $ is connected. The digraph $ D'_v $ arises by deleting at most $ l<k $ edges of the $ k $-edge-connected digraph $ D_v $ 
(see 
Proposition \ref{Dv}) hence it is connected.
\end{sbiz}\\

 We will prove that if $ D $ is $ l $-edge-connected, then it is also $ l+1 $ edge-connected. This 
is enough since we have already proved  $ 1 $-connectivity of $ D $ in Proposition  \ref{D connected}.
Assume that $ D $ is $ l $-edge-connected. Let $ C\subseteq A,\ \left|C\right|=l $ arbitrary and $ D'\defeq 
(V,A\setminus C) $. By the definition of $ l+1 $-edge connectivity we need to show that $ D' $ is connected.
Suppose for contradiction that it is not. Since the connectivity of the subgraphs $ D'[V_i]\ (i\in I) $ implies the connectivity of $ 
D' $ (by Proposition \ref{oroklod})  there is an $ i_0\in I $ such that $ D'[V_{i_0}] $ is not connected. Since the connectivity 
of the subgraphs $ D'[V_{i_0 i}]\ (i\in I) $ implies the connectivity of $ D'[V_{i_0}] $  there is an $ 
i_1\in I $ such that $ D'[V_{i_0 i_1}] $ is not connected$ \dots $ By recursion we obtain an infinite sequence $ (i_n)_{n\in 
\mathbb{N}} $ 
such that the digraphs $ D'[V_{i_0\dots i_n}]\ (n\in \mathbb{N}) $  are all disconnected. Note that the digraphs $ 
D[V_{i_0\dots i_n}]\ (n\in \mathbb{N}) $ are $ l $-connected because $ 
D $  is $ l $-connected by assumption and they are isomorphic to it, hence necessarily $ C \subseteq \mathsf{span}(V_{i_0\dots i_n}) 
$ 
for all $ 
n\in \mathbb{N} $. But then
\[ C\subseteq \bigcap_{n=0}^{\infty} \mathsf{span}(V_{i_0 \dots  i_n})=\mathsf{span}\left( \bigcap_{n=0}^{\infty}  V_{i_0\dots 
i_n}\right)=\mathsf{span}(\varnothing)=\varnothing\]
which is a contradiction since $ \left| C\right|= l\geq 1 $.
\end{biz}
\begin{lem}\label{no pair-connected}
There are no edge-disjoint back and forth paths between $ s $ and $ t $ in $ D $.
\end{lem}
\begin{biz}
Suppose, seeking a contradiction, that there are. Let $P_{s,t} $ be an $ s \rightarrow t $ path and $ P_{t,s} $ be a $ t 
\rightarrow 
s$ 
path such that they are edge-disjoint and have a minimal sum of lengths among these path pairs. For $ u,v\in V $ call a set $ 
U\subseteq V $ an  $ uv $-cut iff $ u\in U 
$ and $ v\notin U $. The set $ \{ t \}\cup \bigcup \{ V_i: i\in I_e \} $ is a $ ts $-cut and its outgoing edges are $ \{ 
(t_i,s_{i+1}) 
\}_{i\in I_e} $. Let $i_0\in I_e $ be the maximal index such that $ P_{t,s} $ uses the edge $ (t_{i_0}s_{i_0+1}) $. Then an 
initial 
segment of $ 
P_{t,s} $ is necessarily of the form $ t,P_{s_{2k-2},t_{2k-2}},P_{s_{2k-4},t_{2k-4}},\dots,P_{s_{i_0},t_{i_0}},s_{i_0+1} $ where 
$ 
P_{s_i,t_i} $ is an $ s_i\rightarrow t_i $ path in $ D[V_i] $. The set 
$T\defeq \{ t \}\cup \bigcup \{ V_i : i_0\leq i\in I \}  $ is also a $ ts $-cut and all the tails of its outgoing 
edges 
are 
in $ \{ t_{i_0},t_{i_0+1} \} $. $ P_{t,s} $ has already used the edge $ (t_{i_0},s_{i_0+1}) $ so it may not use 
another 
edge with tail $ t_{i_0} $ hence $ P_{t,s} $ leave $ T $ using an edge with tail $ t_{i_0+1} $. But then $ 
P_{t,s} $ 
contains 
an $s_{i_0+1}\rightarrow t_{i_0+1} $ subpath $ P_{s_{i_0+1},t_{i_0+1}} $ in $ D[V_{i_0+1}] $.\\
$S\defeq\{ s \}\cup \bigcup \{V_i: i_0+1\geq i\in I \} $ is an $ st $-cut and all the tails of its outgoing edges are 
in 
$ \{ s_{i_0},s_{i_0+1} \} $. Therefore $ P_{s,t} $ has an initial segment  in $ D[S] $ that terminates in this set. We know that $ 
P_{s,t} $ 
does not use the edge $ (t_{i_0},s_{i_0+1}) $ because $ P_{t,s} $ has already used it. Therefore there is an $ m\in \{ i_0,i_0+1 \} $ 
such 
that $ P_{s,t} $ has a  $ t_m\rightarrow s_m $ subpath $ P_{t_m,s_m} $ in $ D[V_m] $. But then the paths $ P_{t_m, s_m} $ and 
$ P_{s_m,t_m} $ are proper subpaths of $ 
P_{s,t} $ and $ P_{t,s} $ respectively. By Proposition \ref{isomorph}  $ f_m $ is an isomorphism  between $ D $ and 
$D[V_m] $ and thus the 
inverse-images of the paths $ P_{t_m, s_m} $ and $ P_{s_m,t_m} $ are edge-disjoint back and forth paths between $ s 
$ and $ t $ 
with strictly less sum of lengths than the added length of paths $ P_{s,t}$ and $P_{t,s} $, which contradicts with the choice 
of $ P_{s,t} $ and $ P_{t,s} $.\end{biz}\\

\end{proof}

\end{document}